\documentclass[reqno]{amsart}

\usepackage{amsfonts}
\usepackage{amsthm}
\usepackage{amssymb}
\usepackage{mathrsfs}
\usepackage{hyperref}
\usepackage{mathtools}
\usepackage[all,cmtip]{xy}
\usepackage{enumitem}  
\usepackage{cleveref} 
\usepackage{cancel}

\usepackage{xcolor}

\theoremstyle{plain}
\newtheorem{theorem}{Theorem}[section]
\newtheorem{definition}[theorem]{Definition}
\newtheorem{proposition}[theorem]{Proposition}
\newtheorem{lemma}[theorem]{Lemma}
\newtheorem{corollary}[theorem]{Corollary}

\def\op{\operatorname}
\def\C{\mathbb{C}}
\def\Bl{\op{Bl}}

\def\E{\mathscr{E}}
\def\F{\mathscr{F}}

\def\K{\mathscr{K}}
\def\I{\mathscr{I}}
\def\J{\mathscr{J}}
\def\O{\mathscr{O}}
\def\T{\mathcal{T}}
\def\L{\mathscr{L}}
\def\N{\mathscr{N}}
\def\P{\mathbf{P}}
\def\H{\mathcal{H}}
\def\Proj{\op{Proj}}
\def\Pic{\op{Pic}}
\def\Ext{\op{Ext}}

\def\sExt{\mathscr{E}\kern -.5pt xt}
\newcommand*{\sHom}{\mathscr{H}\kern -.5pt om}
\renewcommand{\tilde}{\widetilde}

\def\APic{\op{APic}}

\def\GDiv{\op{GDiv}}
\def\GPic{\op{GPic}}
\def\ACart{\op{ACart}}
\def\Cart{\op{Cart}}
\def\Ann{\op{Ann}}
\def\Fitt{\op{Fitt}}
\def\Sym{\op{Sym}}

\begin{document}

\title{Space curves on surfaces with ordinary singularities}
\author{Mengyuan Zhang}
\address{Department of Mathematics, University of California, 
Berkeley, CA 94720}
\email{myzhang@berkeley.edu}

\begin{abstract}
We show that smooth curves in the same biliaison class on a hypersurface in $\P^3$ with ordinary singularities are linearly equivalent.
We compute the invariants $h^0(\I_C(d))$, $h^1(\I_C(d))$ and $h^1(\O_C(d))$ of a curve $C$ on such a surface $X$ in terms of the cohomologies of divisors on the normalization of $X$.
We then study general projections in $\P^3$ of curves lying on the rational normal scroll $S(a,b)\subset\P^{a+b+1}$.
If we vary the curves in a linear system on $S(a,b)$ as well as the projections, we obtain a family of curves in $\P^3$.
We compute the dimension of the space of deformations of these curves in $\P^3$ as well as the dimension of the family.
We show that the difference is a linear function in $a$ and $b$ which does not depend on the linear system.
Finally, we classify maximal rank curves on ruled cubic surfaces in $\P^3$. 
We prove that the general projections of all but finitely many classes of projectively normal curves on $S(1,2)\subset\P^4$ fail to have maximal rank in $\P^3$. 
These give infinitely many classes of counter-examples to a question of Hartshorne \cite{H2}.
\end{abstract}

\maketitle

\section*{Introduction}
The study of algebraic curves is a central topic of algebraic geometry.
Since every smooth projective curve embeds in $\P^3$, a lot of research has been devoted to the study of space curves.
An important way to produce examples of space curves is to use high degree divisors on the curves to embed them in large dimensional projective spaces, and then project them isormophically into $\P^3$.
Another important perspective is to realize curves as divisors on an ambient surface.
In this paper, we combine both methods and study curves in $\P^3$ lying on general projections of a smooth surface in higher dimensional projective space.

\bigskip

For the setup, let $\pi:S\to X$ be a dominant finite birational map from a smooth projective surface $S$ to a hypersurface $X\subset\P^3$.
We assume $\pi$ is not an isomorphism and satisfies certain genericity assumptions, such as the general projections of smooth surfaces into $\P^3$.
Two curves $C$ and $D$ on $X$ are linked if the $C$ and $\O_X(m)-D$ are linearly equivalent as generalized divisors on $X$ for some positive $m$.
Curves in the linear equivalence classes of $\O_X(m)+C$ for any $m$ are said to be in the biliaison class of $C$ on $X$.
See \cite{MDP} for background on liaison theory and \cite{H1} for the interpretation of liaison theory via generalized divisors.
We summarize the main results of this paper below.

\bigskip

In \Cref{Section1}, we study the liaison theory of curves on $X$.
One main result is that smooth curves on $X$ in the same biliaison class are unique up to linear equivalence.
More precisely, if $C$ is a smooth curve on $X$ not supported on any components of the singular locus, then no divisor in the class of $\O_X(m)+C$ is smooth if $m\ne 0$.
This is in clear contrast to the situation on smooth surfaces.
Next, we give a criterion of linkage of curves on $X$, and show that certain homological invariants of a curve $C$ on $X$ can be computed in terms of cohomologies of divisors on $S$, provided that $C$ is preserved by the morphism $\pi$ or is linked to such a curve. 
Examples of homological invariants of $C$ include the Hilbert function $h^0(\I_C(n))$, the Rao function $h^1(\I_C(n))$, the specialty function $h^1(\I_C(n))$ and the dimension of the tangent space $h^0(\N_C)$ at the corresponding point in the Hilbert scheme.
In particular, this means that if $S$ is a surface where the cohomology of divisors can be computed easily, such as rational normal scrolls, certain other rational surfaces, toric surfaces etc., then we can determine a lot of information of the general projection of a curve $C$ on $S$ into $\P^3$.
If we vary the curves in a linear system on a surface in a high dimensional projective space and vary the projections, then we obtain an interesting family of curves in $\P^3$.

\bigskip

In \Cref{Section2}, we use the results in \Cref{Section1} to study general projections of curves lying on a rational normal scroll $S(a,b)\subset\P^{a+b+1}$.
We use results of \cite{KLU} and classical theorems on general projections of surfaces to enumerate the singularities of the image surface $X \subset \P^3$.
We compute the dimension of the family of curves in $\P^3$ arising from various projections of curves varying in a given linear system on $S(a,b)$. 
In particular, we show that the difference between the dimension of the tangent space at the corresponding points in the Hilbert scheme and the dimension of the family is a linear function in $a$ and $b$, which does not depend on the linear system chosen.
This generalizes Gruson and Peskine's example of families of curves on a ruled cubic surface in $\P^3$, and gives interesting families of curves in $\P^3$ whose dimensions of the tangent spaces in the Hilbert scheme can be arbitrarily large compared to the dimension of the family.
Last but not least, we determine all maximal rank curves on a ruled cubic surface.
In particular, we show that the general projections of all but finitely many classes of projective normal curves on the cubic scroll $S(1,2)\subset\P^4$ fail to have maximal rank in $\P^3$.
This provides infinitely many classes of counter-examples to a question raised by Hartshorne.

\bigskip

We provide some historical context to the investigations in this article.

\bigskip

In 1882, the Berlin Academy of Sciences awarded the Steiner prize to M. Noether and G. Halphen for their work on the classification of space curves. 
One problem in the classification is to determine all the possible pairs of values of degree and genus for a smooth connected projective curve in $\P^3$.
Except for plane curves and curves on a smooth quadric surface, which are completely understood, Halphen proved that the degree $d$ and genus $g$ of all smooth space curves must satisfy $g\le \frac{1}{3}d(d-3)+1$.
A century later, Gruson and Peskine \cite{GP} showed that all values of $g$ satisfying this inequality can be obtained.
For $\frac{1}{\sqrt{3}}d^{3/2}-d+1\le g\le \frac{1}{6}d(d-3)+1$, curves with this $(d,g)$ can be found on a smooth cubic surface. 
For $0\le g\le \frac{1}{8}(d-1)^2$, curves with this $(d,g)$ can be found on a quartic surface with a double line.
This completes the classification of all possible degree and genus of smooth space curves.

Since the geometry of a smooth cubic surface is well studied, much is known for curves lying on smooth cubic surfaces.
Notably, Mumford \cite{Mumford} showed that there is a generically nonreduced component in the Hilbert scheme $\H^3_{14,24}$ of degree 14 genus 24 curves in $\P^3$, whose general point corresponds to a curve lying on a smooth cubic surface. 
Kleppe \cite{Kleppe} carried out a systematic study of curves on smooth cubic surfaces and extended Mumford's results in certain cases.
On the other hand, less is known for curves on singular surfaces, especially those with one-dimensional singular locus.
Since such surfaces are no longer normal, the theory of Weil divisors, which is a cornerstone of the study of surfaces, breaks down.
To remedy this problem, Hartshorne \cite{H1} developed the theory of generalized divisors on a Gorenstein scheme.
Since hypersufaces are Gorenstein, this allows us to discuss curves on any surface in $\P^3$, even if the surface is reducible or nonreduced.
In the same paper, Hartshorne carried out a detailed study of curves on the ruled cubic surface in $\P^3$.
Using these results and bounds on the dimension of the Hilbert scheme in certain cases, Hartshorne \cite{Nonsmooth} found the first example of an irreducible curve in $\P^3$ that is not smoothable.

A recent advance in the study of algebraic curves is the proof of the longstanding maximal rank conjecture by E. Larson \cite{Larson}.
A projective variety $V\subset \P^n$ is said to have maximal rank if the maps of vector spaces $H^0(\O_{\P^n}(d))\to H^0(\O_V(d))$ have maximal rank for all $d\ge 0$.
This is equivalent to the condition that $H^0(\I_{V}(d)) = 0$ or $H^1(\I_V(d)) = 0$ for all $d\ge 0$.
Severi in 1915 conjectured that a general embedding of a general curve of genus $g$ has maximal rank. 
This conjecture can be made precise using Brill-Noether theory, developed by Griffiths-Harris \cite{BN}, Gieseker \cite{Gieseker}, Kleiman-Laksov \cite{KL} and others. 
As a consequence of the maximal rank conjecture, the Hilbert functions of general curves under general embeddings are completely determined.
A related question was asked by Hartshorne in \cite{H2}:
If $C\subset\P^N$ is a projectively normal curve, then does a sufficiently general projection of $C$ in $\P^3$ have maximal rank?
Counter-examples have been found by Gruson-Peskine \cite{GP2}.
On the other hand, Ballico and Ellia \cite{BE} proved that the statement is true if $0\le g\le 3$ and the degree $d$ of the curve is large compared to its genus $g$.
Complementary to this result, the counter-examples provided in this paper can have arbitrarily large genus $g$, and the degree $d$ is small compared to $g$.
In view of the maximal rank theorem, these are embeddings of curves that are ``general" (in the sense of having maximal rank), but whose general projections in $\P^3$ are ``special". 

\subsection*{Acknowledgement}
The author thanks Hartshorne for initiating his interest in space curves and for suggesting to study curves on ruled cubic surfaces.

\setcounter{section}{-1}
\section{Background}
We work over the field of complex numbers $\C$.
By a curve we mean a one-dimensional projective scheme without point components, isolated or embedded.
By a point on a finite type $\C$-scheme we mean a closed point.
We denote the arithmetic genus of a curve $C$ by $g(C)$. 

\bigskip

In this section, we review the multiple-point formulas developed by Kleiman \cite{Kleiman} and Kleiman-Lipman-Ulrich \cite{KLU}.
These formulas enumerate the singularities of the images of the projection.
We also review the classical theorems on the singularities of general projections of smooth surfaces into $\P^3$.
Last but not least, we review the theory of generalized divisors developed by Hartshorne \cite{H1} in order to discuss divisors on singular surfaces.

\subsection{Multiple point formulas}\

We fix a finite morphism $f:X\to Y$ of finite type $\C$-schemes.

\begin{definition}[Multiple-point and ramification loci]
Let $N_r$ denote the subset of points $y \in Y$ such that $f^{-1}(y):= X\times_Y k(y)$ contains a zero dimensional subscheme of length $r$.
By fiber continuity, the subset $N_r$ is closed in $Y$. 
Let $M_r := f^{-1}(N_r)$ denote the closed subset in $X$.
We call $M_r$ and $N_r$ the source and target $r$-fold points of $f$ respectively.
\end{definition}

The central theme of the theory of multiple-point formulas is to find appropriate scheme structures on $M_r$ and $N_r$, and to determine the classes $[M_r]$ and $[N_r]$ in the Chow ring (or other cohomology rings).
The subject started in 1850 and still inspires current research.
For example, such a formula can give the degree and genus of the curve of trisecant lines of a given space curve $C$, as well as the number of quadrisecant lines.
See \cite[V]{Kleiman} for a survey on the subject of multiple-point formulas.
See \cite{Kleiman1} for an approach using iterations, and \cite{Kleiman2} for an approach using Hilbert schemes. 

\bigskip

In codimension one, Kleiman-Lipman-Ulrich \cite{KLU} considered the subscheme structures on $N_r$ given by the Fitting ideals $\op{Fitt}_{r-1}^{\O_Y}(f_* \O_X)$ and the corresponding preimage subscheme structures on $M_r$.
Under mild assumptions, the subschemes $M_r$ and $N_r$ are Cohen-Macaulay, and their classes $[M_r]$ and $[N_r]$ are compatible with those coming from iteration \cite{Kleiman1}.
We summarize here some of the results of \cite{Kleiman1} and \cite{KLU}.

\begin{definition}
With notations as above, let $R_i$ be the subscheme of $X$ defined by the fitting ideal $\Fitt_{i-1}^{\O_X}(\Omega_{X/Y})$. 
By the conormal sequence $f^* \Omega^1_{Y/\mathbb{C}} \to \Omega^1_{X/\C} \to \Omega_{X/Y} \to 0$, 
the scheme $R_i$ is supported at points in $X$ where the differential $\partial f$ drops rank by at least $i$. 
We call $R := R_1$ the ramification locus. 
\end{definition}

\begin{theorem}[Kleiman-Lipman-Ulrich]\label{MultiplePoint}\
\begin{enumerate}[label=(\roman*)]
\item Suppose $f$ is locally of flat dimension 1 and birational onto its image, then $N_1$ is equal to the scheme-theoretical image $f(X)$ and $f_*[M_1] = [N_1]$. 
\item Suppose furthermore that $Y$ satisfies ($S_2$).
The ideal sheaf $\I_{N_2/Y}$ is equal to $\Ann_{\O_{Y}}(f_*\O_X/\O_{Y})$ and the ideal sheaf $\I_{N_2/N_1}$ is equal to $\Ann_{\O_{N_1}}(f_*\O_X/\O_{N_1})$.
Each component of $M_2$ has codimension 1
and maps onto a component of $N_2$.
Each component of $N_2$ has codimension 2.
The fundamental cycles of these two schemes are related by the equation
$f_*[M_2] = 2[N_2]$.
The $\O_Y$-modules $\O_{N_2}$ and $f_*\O_{M_2}$ are perfect of grade 2.
\item For $r>0$, suppose furthermore that $Y$ satisfies ($S_r$) and $R_2 = \varnothing$. 
If each component of $N_r$ has codimension $r$, then $\O_{N_r}$ and 
$f_*\O_{M_r}$ are perfect $\O_Y$-modules of grade $r$.
Each component of $M_r$ has codimension
$r-1$ and maps onto a component of $N_r$, and the fundamental cycles of these two
schemes are related by the equation $f_*[M_r] = r[N_r]$. 
\end{enumerate}
Suppose $f$ satisfies all the assumptions above for all $r>0$.
If $f$ is a local complete intersection and $X$ has no embedded components, then 
\[
[M_{r+1}] = f^*f_* [M_r] - r c_1(\nu)[M_r].
\]
Here $\nu = f^* \T_Y - \T_X$ is the virtual normal bundle in the Grothendieck group $K(X)$.
\end{theorem}

\subsection{General linear projections of smooth surfaces into $\P^3$}\

Every smooth projective surface $S\subset \P^N$ for $N>5$ can be projected isomorphically into $\P^5$, but not further down in general.
The following is a summary of the classical results on the singularities of general linear projections of $S$ into $\P^3$.

\begin{theorem}\label{Projection}
Let $S\subset \P^5$ be a non-degenerate smooth projective surface that is not the Veronese surface, then the following are true for a general linear projection $f:S\to \P^3$.
\begin{enumerate}
  	\item The map $f$ is birational onto its scheme-theoretical image $X$, which is an integral hypersurface in $\P^3$.
 	\item The second ramification locus $R_2$ and the quadruple-point loci $M_4$ and $N_4$ are empty.
    \item The target double-point locus $N_2$ is an integral curve and is exactly the singular locus of $X$.
    The target triple-point locus $N_3$ is a reduced set of points and is exactly the singular locus of $N_2$.
    Each point of $N_3$ is an ordinary triple point of $N_2$ and of $X$.
    
    \item The source double-point locus $M_2$ is an integral curve mapping generically 2-1 to $N_2$.
    The source triple-point locus $M_3$ is a reduced set of points and is exactly the singular locus of $M_2$.
    Each point of $M_3$ is a simple node of $M_2$.
    Three distinct points of $M_3$ map to each point of $N_3$.
    \item The ramification locus $R$ is a set of reduced points and consisting exactly of the ramification points of the double cover $M_2\to N_2$.
    The images of points in $R$ are exactly the pinch points of $X$. 
    \item Every point of $N_2-N_3-f(R_1)$ is an ordinary double point of $X$.
\end{enumerate}
If $S$ is the Veronese surface, then a general linear projection of $S$ in $\P^3$ is called a Steiner surface. 
In this case, the scheme $N_2$ is the union of three reduced lines $L_1,L_2,L_3$ meeting at a point $p$.
The scheme $M_2$ is three reduced conics $C_1,C_2,C_3$ meeting each other at one point, where the three intersection points all map to $p$.
The conic $C_i$ maps 2-1 to $L_i$ with two ramification points. 
The six branch points are the pinch points of $X$.
In particular, all the results above are true except that $M_2$ and $N_2$ are reducible.
\end{theorem}
\begin{proof}
See \cite{Harris}, \cite{Tour} and \cite{JR1} for expositions of these classical facts. 
\end{proof}

The following definition is made for surfaces whose singularities resemble those occurring on the general linear projections of smooth surfaces in $\P^3$.

\begin{definition}
An integral hypersurface $X\subset \P^3$ is said to have ordinary singularities if the singular locus of $X$ is a curve $C$, such that if we put the reduced structure on $C$ then the following hold.
\begin{enumerate}
\item Singular points of $C$ are ordinary triple points (i.e. the origins of three linear branches of $C$ with three distinct non-coplanar tangent directions).
\item  
A point on $C$ is either a nodal point of $X$ (i.e. the origins of two linear branches of $X$ with two distinct tangent planes) or a pinch point of $X$ (i.e analytically isomorphic to $x^2-yz = 0$ at origin), and there are only finitely many pinch points on $X$.
\item Every triple point of $C$ is an ordinary triple point of $X$ (analytically isomorphic to $xyz = 0$ at origin).
\end{enumerate}
\end{definition}

An integral hypersurface $X \subset \P^3$ with ordinary singularities has a smooth normalization.
The normalization may not be embedded in a projective space of which $X$ is a linear projection, i.e. the pull back of the $\O_X(1)$ may not be very ample.
Such an example is given by the quartic surface with a double line considered by Gruson-Peskine \cite{GP} in order to determine all possible pairs of values of degree and genus for smooth space curves.

\bigskip

Since $S$ is smooth and $\P^3$ is Cohen-Macaulay, \Cref{Projection} implies that a general linear projection $f:S\to \P^3$ satisfies all the assumptions in \Cref{MultiplePoint}.

\begin{corollary}\label{Formulas}
Let notations be as above.
Denote by $h$ the class of $f^*\O_{\P^3}(1)$ in the Chow ring $A(S)$, and denote by $c_i$ the Chern classes of $\T_S$.
For two divisor classes $a,b\in A^1(S)$, let $a.b$ denote the intersection number.
The following hold.
\begin{enumerate}[label = (\alph*)]
\item The class of $M_2$ in $A^1(S)$ is $(h.h-4)h+c_1$. \label{M2}
\item The degree of the curve $N_2$ is $\frac{1}{2}((h.h)^2-4h.h+c_1.h)$. \label{dN2}
\item There are short exact sequences
\begin{align}
& 0 \to \O_X \to f_* \O_S \to \omega_{N_2}(4-h.h) \to 0, \tag{$A$} \label{E1}\\
& 0 \to \O_{N_2} \to f_* \O_{M_2} \to \omega_{N_2}(4-h.h) \tag{$B$}\label{E2} \to 0.
\end{align} \label{SES}
\item The arithmetic genus of $N_2$ is equal to
\[
\frac{1}{3}(h.h)^3- 3(h.h)^2+\frac{37}{6} (h.h) + \frac{1}{2} (h.h)(h.c_1)-2(h.c_1)+\frac{1}{12}(c_1.c_1+c_2.1) +1.
\] 
\label{gN2}
\item The class of $M_3$ in $A_0(S)$ is equal to 
\[
((h.h)^2-12h.h+c_1.h+44)h^2+(2h.h-24)c_1h+4c_1^2- 2c_2.
\] 
The number of triple points of $X$ is one third the degree of $[M_3]$. \label{M3}
\item The class of $R_1$ in $A_0(S)$ is equal to $6h^2-4hc_1+c_1^2-c_2$.
The number of pinch points of $X$ is the degree of $R_1$. \label{R1}
\end{enumerate}
\end{corollary}
\begin{proof}
\noindent\ref{M2} $[M_2] = f^*f_*[M_1]-c_1(\nu) = (h.h-4)h+c_1$.

\noindent\ref{dN2} By push-pull formula, we have $M_2.h = 2\deg N_2 = (h.h-4)h.h+c_1.h$.

\noindent\ref{SES} The $\O_X$-dual of the exact sequence 
\[
0 \to \I_{N_2/X} \to \O_X \to \O_{N_2} \to 0
\]
gives an exact sequence
\[
0 \to \O_X \to \sHom_{\O_X}(\I_{N_2/X},\O_X) \to \sExt^1_{\O_X}(\O_{N_2},\O_X) \to 0.
\]
Note that $\I_{N_2/X} = \sHom_{\O_X}(f_* \O_S,\O_X)$, and that $f_*\O_S$ is a reflexive $\O_X$-module since it satisfies (S2) and $X$ is Gorenstein.
It follows that we have isomorphisms 
\[
\sHom_{\O_X}(\I_{N_2/X},\O_X) = \sHom_{\O_X}(\sHom_{\O_X}(f_* \O_S,\O_X),\O_X) = f_*\O_S.
\]
Since $N_2$ is Cohen-Macaulay, the third term in the exact sequence is isomorphic to 
\[
\sExt^1_{\O_X}(\O_{N_2},\O_X)\cong \omega_{N_2}\otimes  \omega_X^{-1} \cong \omega_{N_2}(4-h.h).
\]
We obtain the first exact sequence. 
Since $f_*\I_{M_2/S} = \I_{N_2/X}$, and $R^if_* = 0$ for $i>0$ because $f$ is affine, the snake lemma gives the second exact sequence. 
These two sequences are due to Roberts \cite{JR}.

\noindent\ref{gN2} This is a computation derived from sequence (\ref{E1}).
\begin{align*}
g(N_2) & = 1-\chi(\O_{N_2}) \\
& = 1+\chi(\omega_{N_2})\\
& = 1+\chi((f_*\O_S)(h.h-4))-\chi(\O_{N_1}(h.h-4))\\
& = 1+ \chi(O_S((h.h-4)h)) - \chi(\O_{\P^3}(h.h-4)) +\chi(\O_{\P^3}(-4))\\
& = 1+\frac{1}{2}(h.h-4)h.((h.h-4)h+c_1)+\frac{1}{12}(c_1.c_1+c_2.1) - \frac{1}{6}(h.h-1)(h.h-2)(h.h-3)-1\\
& = \frac{1}{3}(h.h)^3- 3(h.h)^2+\frac{37}{6} (h.h) + \frac{1}{2} (h.h)(h.c_1)-2(h.c_1)+\frac{1}{12}(c_1.c_1+c_2.1) +1
\end{align*}
In the second line, we used the fact that $N_2$ is Cohen-Macaulay and Serre duality holds. In the third line, we used the exact sequence (\ref{E1}). In the fourth line, we use the projection formula to conclude that 
\[
h^i((f_*\O_S)(h.h-4)) = h^i(f_*(\O_S((h.h-4)h))) = h^i(\O_S((h.h-4)h)
\]
since $\pi$ is affine.
Since $N_1$ is a hypersurface of degree $h.h$, there is an exact sequence
\[
0 \to \O_{\P^3}(-4) \to \O_{\P^3}(h.h-4) \to \O_{N_1}(h.h-4) \to 0.
\]
On the next line, we applied Hirzebruch-Riemann-Roch on the smooth surface $S$ to the line bundle $(h.h-4)h$ and used the fact that $\chi(\O_{\P^3}(d)) = \frac{1}{6}(d+3)(d+2)(d+1)$. 
It is somewhat curious that this expression always ends up being an integer.

\noindent\ref{M3} The triple-point formula yields
\begin{align*}
[M_3] & = f^*f_*[M_2] - 2c_1(\nu)[M_2]+2c_2(\nu) \\
& =((h.h)^2-12h.h+c_1.h+44)h^2+(2h.h-24)c_1h+4c_1^2- 2c_2.
\end{align*}
The number of triple points is the degree of $[N_3]$, which is one third the degree of $[M_3]$ since $f_*[M_3] = 3[N_3]$.

\noindent\ref{R1}  Applying Porteous formula to the transpose of the map $\partial f: f^*\Omega_{\P^3/\C} \to \Omega_{S/\C}$ yields the result.
Since ramification points map bijectively to pinch points, the number of pinch points is given by the degree of $[R_1]$.
\end{proof}

\subsection{Generalized divisors}\

In this subsection, we fix $\pi:S\to X$ a finite birational morphism from a smooth projective surface to an integral singular hypersurface in $\P^3$.
Let $h$ denote the class of $\pi^* \O_X(1)$ in $A^1(S)$.
We briefly review the definitions and basic operations of generalized divisors.

\begin{definition}
Generalized divisors on a Gorenstein scheme $X$ are reflexive fractional ideals of rank one. In particular, curves on Gorenstein surfaces are exactly the effective generalized divisors. An almost Cartier divisor is a generalized divisor that is locally principal away from a closed subset of codimension at least two. Let $\Cart$, $\ACart$ and $\GDiv$ denote the group of Cartier divisors, the group of almost Cartier divisors and the set of generalized divisors. Let $\Pic$, $\APic$ and $\GPic$ be the corresponding isomorphism classes.
\end{definition}

For example, any curve not supported on any components of the singular locus of a Gorenstein surface in dimension zero is an almost Cartier divisor, but may not be Cartier.

\bigskip

There is a pullback map $\pi^*:\ACart(X) \to \ACart(S) = \Cart(S)$ defined by sending a reflexive fractional ideal $\I$ of rank one to $((\O_S\cdot \I)^{-1})^{-1}$. Here $\J^{-1}$ is defined to be $(\O_S:_{\K_S} \J)$ for a fractional ideal $\J$, and is isomorphic to the dual $\J^*$. One can check that $\pi^*$ is a morphism of groups and descends to $\pi^*:\APic(X)\to \Pic(S)$.  

\begin{theorem}[Hartshorne-Polini{\cite[Thm 4.1]{HP}}]\
\begin{enumerate}
\item There is a morphism of groups $\varphi:\APic(X)\to \Cart M_2/ \pi^* \Cart N_2$.
\item There is an exact sequence of groups
\[
0 \to \APic(X) \to \Pic(S)\oplus \Cart M_2/\pi^* \Cart N_2 \to \Pic M_2/\pi^* \Pic N_2 \to 0.
\]
The first map of the short exact sequence is given by $\pi^* \oplus \varphi$ and the second map is given by the difference of the two maps 
\[
\Pic(S)\to \Pic M_2/\pi^* \Pic N_2 \text{ and }\Cart M_2/\pi^* \Cart N_2 \to \Pic M_2/\pi^* \Pic N_2.
\]
\end{enumerate}
\end{theorem}

The map $\varphi:\APic(X)\to \Cart M_2/\pi^* \Cart N_2$ is important, so we review how it is defined. Let $D$ be an almost Cartier divisor, then $[D] =[C_1]-[C_2]$ for two almost Cartier divisors $C_1$ and $C_2$ not supported on any components of $N_2$ by \cite[Prop 2.11]{H1}. Then $\tilde{C_1}$ and $\tilde{C_2}$ are two Cartier divisors on $S$ intersecting $M_2$ properly, and thus restricts to Cartier divisors $\alpha_1 = \tilde{C_1}\cap M_2 |_{M_2}$ and $\alpha_2 = \tilde{C_2}\cap M_2 |_{M_2}$ on $M_2$. We define $\varphi [D]$ to be the image of $\alpha_1-\alpha_2$ in $\Cart M_2/N_2$. 
This map is well-defined by \cite[Prop 2.3]{HP}.

\bigskip

\section{Curves on hypersurfaces with ordinary surface singularities}\label{Section1}

In this section, we study the liaison theory of curves on $X$ utilizing its normalization $S$.
We pay particular attention to the homological invariants of these curves $C$, such as the specialty function $h^1(\O_C(n))$, the Hilbert function $h^0(\I_{C/\P^3}(n))$ and the Rao function $h^1(\I_{C/\P^3}(n))$. 

\subsection{Linkage on a hypersurface with ordinary singularities}\

\begin{definition}[Linkage]
Two curves $C$ and $D$ on $X$ are linked if $D$ and $\O_X(m)-C$ are linearly equivalent as generalized divisors on $X$ for some positive integer $m$. 
\end{definition}

This generalizes the classical notion of linkage in the classical sense.
See \cite[\S 4]{H1} for a treatment of linkage theory in terms of generalized divisors. 

\begin{proposition}
The map $\pi:S\to X$ is the blowup of $X$ with center $N_2$. 
\end{proposition}
\begin{proof}
Recall that $\I_{N_2/X}$ is the conductor of the normalization and thus $\pi_* \I_{M_2/S} = \I_{N_2/X}$. There is an isomorphism of schemes: 
\[
\Bl_{N_2} X = \Proj \bigoplus_{i = 0}^\infty \I_{N_2/X}^i \cong \Proj \bigoplus_{i = 0}^\infty \I_{M_2/S}^i= \Bl_{M_2} S.
\]
But $\Bl_{M_2} S = S$ since $M_2$ is a divisor.
\end{proof}

If $C$ is a curve on $X$ not supported on any components of $N_2$, then we reserve the notation $\tilde{C}$ for the proper transform of $C$ on $S$. 

\begin{proposition}
If $C$ is a curve on $X$ not supported on any components of $N_2$, then $\pi^* C$ and $\tilde{C}$ are the same Cartier divisor on $S$. 
\end{proposition}
\begin{proof}
Recall that any reflexive sheaf over $S$ extends uniquely from an open set with a codimension 2 complement \cite[Prop 1.11]{H1}. Since $\pi^{-1}(C\cap N_2)$ is codimension two in $S$ and $\tilde{C}$ is equal to $\pi^*C$ on the open set $S-\pi^{-1}(C\cap N_2)$, it follows that $\tilde{C}$ is equal to $\pi^* C$ on the whole of $S$.
\end{proof}

\begin{corollary}\label{degree}
If $D$ is an effective almost Cartier divisor on $X$, then $\pi_*\pi^*[D] = [D]$ as Chow classes. In particular $(\pi^*[D]).h = \deg D$. 
\end{corollary}
\begin{proof}
Since $D$ is almost Cartier, by \cite[Prop 2.11]{H1} $D$ is linearly equivalent to $C_1-C_2$, where $C_1,C_2$ are effective almost Cartier divisors not supported on any components of $N_2$. Then $[D] = [C_1]-[C_2]$ as Chow classes. Now $\pi_*\pi^*[C_i] =\pi_*[\tilde{C_i}] = [C_i]$ since $\tilde{C_i} \to C_i$ is degree one on every component. It follows that $\pi_*\pi^*[D] = [D]$, and the push-pull formula yields the last claim. 
\end{proof}

If $N_2$ were almost Cartier on $X$, then $\pi_*\pi^*[N_2] = [N_2]$ by the above proposition.
However, since $\O_S\cdot \I_{N_2/X} = \I_{M_2/S}$, we must have $\pi^*[N_2] = [M_2]$. This contradicts the fact that $f_*[M_2] = 2[N_2]$.
Therefore the generalized divisor $N_2$ is not almost Cartier on $X$. 
Consequently neither is $N_2+C$ for any almost Cartier divisor $C$. 

\begin{lemma}\label{cartier}
If $D$ is a Cartier divisor on $X$ not supported on any components of $N_2$, then 
$\pi^*(D\cap N_2) = (\pi^*D)\cap M_2 =  \tilde{D}\cap M_2$
as Cartier divisors of $M_2$.  
\end{lemma}
\begin{proof}
Suppose $D$ is defined locally by $(U_i,f_i)$.
Since $D$ is not supported on $N_2$, the local sections $f_i$ restrict to non-zerodivisors in $H^0(U_i\cap N_2,\K_{N_2})$. 
The corresponding local sections $(\pi^{-1}(U_i),\pi^\# f_i)$ of $\K_S$ define a Cartier divisor $\pi^*D$ not supported on $M_2$, and thus restrict to non-zerodivisors in $H^0(\pi^{-1}(U_i)\cap M_2, \K_{M_2})$. The three Cartier divisors on $M_2$ are equal because they are all defined by the data $(\pi^{-1}(U_i)\cap M_2,\pi^\# f_i)$.
\end{proof}

\begin{proposition}\label{linkage}
Let $C$ and $D$ be two curves on $X$ not supported on any components of $N_2$, then $C$ and $D$ are linked by $\O_X(m)$ if and only if 
\begin{enumerate}
\item $[\tilde{C}]+[\tilde{D}] = mh$, where $h$ is the class of $\pi^*\O_X(1)$;
\item $\tilde{C}\cap M_2 + \tilde{D} \cap M_2$ is a Cartier divisor of $M_2$ in $\pi^* \Cart N_2$.
\end{enumerate}
\end{proposition}
\begin{proof}
Suppose $C$ and $D$ are linked by $\O_X(m)$, where neither are supported on any component of $N_2$. Since both are almost Cartier divisors and the pullback of almost Cartier divisors is a group homomorphism, we must have (1). Lemma \ref{cartier} implies we must have (2). Conversely, $C$ and $D$ satisfy (1) and (2). For any codimension 2 point $x\in X$, the map $\varphi_x:\APic(\O_{X,x})\to\Cart M_{2,x}/\pi^* \Cart N_{2,x}$ is an isomorphism by \cite[Thm 3.1]{HP}. Since (2) is satisfied, $C+D$ has image 0 in $\APic(\O_{X,x})$ for every codimension 2 point $x\in X$ and thus $C+D$ is Cartier. Since both $C+D$ and $\O_X(m)$ pull back to $mh$, we conclude that they are isomorphic since the map $\pi^*: \Pic(X) \to \Pic(S)$ is injective over the complex numbers by \cite[Thm 4.5]{HP}. 
\end{proof}

\begin{theorem}
Suppose $\pi:S\to X$ is induced by a general linear projection $f:S\to \P^3$ of a nondegenerate smooth surface $S$ in $\P^5$.
Suppose $C$ is an integral curve on $S$ meeting $M_2$ transversely avoiding $R_1$.
If $m = 2C.M_2/h.M_2$ is a positive integer such that $mH-C$ is effective and $h^0(\O_S(mH-C))> C.M_2$, then there is a curve $D$ on $S$ such that $\pi(C)$ and $\pi(D)$ are linked on $X$.
\end{theorem}
\begin{proof}
Since $C$ is a Cartier divisor on $S$, it meets $M_2$ at a Cartier divisor.
Since $C$ meets $M_2$ transversely, the intersection $C\cap M_2$ is a reduced set of points $p_1,\dots, p_n$ outside $R_1$ and $M_3$ where $n = C.M_2$.
Thus each point $p_i$ has precisely one corresponding reduced point $q_i$ in $f^{-1}f(p_i)$.
It follows that $p_i+q_i = \pi^{-1} \pi(p_i)$ is a Cartier divisor on $M_2$ in the image of $\pi^*\Cart N_2$.
By the proposition above, we need to find an effective divisor $D$ in the class of $mH-C$ which meets $M_2$ at $q_1,\dots, q_n$.
The choice of $m$ guarantees that $D.M_2 = C.M_2 = n$.
This is always possible if $h^0(\O_S(mH-C))>n$ considering the exact sequence
\[
0 \to H^0(\I_{Q}(mH-C)) \to H^0(\O_S(mH-C)) \to H^0(\O_{Q}),
\]
where $Q$ is the subscheme of $S$ consisting of the reduced points $q_1,\dots, q_n$.
\end{proof}

\subsection{Preserved curves}

\begin{definition}[Preserved curves]
A curve $C$ on $X$ is preserved if a curve $C'$ on $S$ maps isomorphically to $C$. 
\end{definition}

\begin{proposition}
Let $C$ be a preserved curve on $X$.
If $M_2$ is irreducible, then so is $N_2$, and $C$ is not supported on $N_2$.
\end{proposition}
\begin{proof}
The map $M_2\to N_2$ is not injective on the topological spaces. If it were, then $\pi:S\to X$ would be a homeomorphism of topological spaces. But both $S$ and $X$ are reduced, which would imply that $\pi$ is an isomorphism, contrary to our assumption. If $C$ contains $N_2$ set-theoretically, then any curve $C'$ on $S$ mapping onto $C$ must contain $M_2$ set-theoretically, and thus cannot be mapped isomorphically to $C$. 
\end{proof}

Preserved curves may be supported on $N_2$ when $M_2$ is not irreducible. 
If $X$ is the union of two planes meeting at a line $L = N_2$ in $\P^3$, and $S$ is its normalization given by the disjoint union of two planes, then $L$ is preserved since it is the isomorphic image of any one of the two lines on $S$. 

\begin{proposition}
Let $C$ be a curve on $X$ not supported on any component of $N_2$, then $C$ is preserved iff $\tilde{C}\to C$ is an isomorphism iff $C\cap N_2$ is a Cartier divisor on $C$. In particular, smooth curves on $X$ not supported on any components of $N_2$ are preserved. 
\end{proposition}
\begin{proof}
Since $C$ is not supported on any components of $N_2$, the proper transform $\tilde{C}$ on $S$ is isomorphic to the blowup of $C$ at $C\cap N_2$ by the universal property of the blowup. If $C$ is a preserved curve, then the curve on $S$ that maps isomorphically to $C$ must be its proper transform $\tilde{C}$. 
It follows that $\tilde{C}$ is isomorphic to $C$ iff $C\cap N_2$ is a Cartier divisor on $C$. 
\end{proof}

\begin{theorem}\label{unique}
For an almost Cartier curve $C$ on $X$, we define the rational number
\[
m := \frac{2 (\pi^*C).M_2-2(g(C)-g(\pi^*C))}{h.M_2}.
\] 
Here the genus of a divisor is defined by the adjunction formula, which agrees with the arithmetic genus when the divisor is effective.
\begin{enumerate}
\item If $C$ is linked to a preserved curve $D$ by $\O_X(n)$ for some $n>0$, then $n = m$.
\item Conversely, if $m$ is a positive integer, then any nonzero section of $\I_{C/X}(m)$ defines a curve $D$ that is either preserved or is supported on a component of $N_2$. 
\end{enumerate}
\end{theorem}
\begin{proof}
Suppose $C$ is linked to a preserved curve $D$ by $\O_X(n)$ for some $n$. Then the arithmetic genus of $C$ and $D$ are related by 
\begin{align}
g(D)-g(C) = \frac{1}{2}(h^2+n-4)(\deg D-\deg C) = \frac{1}{2}(h^2+n-4)(h^2n-2 \deg C)
\end{align}
using liaison theory, see for example \cite[III Prop.1.2]{MDP}. 
Since $\pi^*: \APic(X)\to \Pic(S)$ is a group homomorphism, we have $\tilde{D}= nh-\pi^*C$ in $\Pic(S)$. 
The adjunction formula on $S$ yields     
\begin{align}
2g(\tilde{D}) -2 = \tilde{D}.(\tilde{D}+c_1) = (nh-\pi^*C).(nh-\pi^*C+c_1).
\end{align}
Since $\tilde{D}\to D$ is an isomorphism, we have $g(\tilde{D}) = g(D)$. 
Combining (1) and (2), we have an equality
\[
(h^2+n-4)(nh^2-2\deg C)+2g(C)-2 = (nh-\pi^*C).(nh-\pi^*C+c_1).
\]
Now $\pi^*C.h = \deg C$ since $C$ is almost Cartier by Proposition \ref{degree} and $M_2 = (h^2-4)h-c_1$ by \Cref{MultiplePoint}. After substitution, we arrive at the linear equation in $n$
\[
nh.M_2  = 2M_2.\pi^*C - 2(g(C) - g(\pi^*C)).
\]
Since $h.M_2 = 2\deg N_2\ne 0$, we see that $n$ must be equal to
\[
n = \frac{M_2.\pi^*C-(g(C)-g(\pi^*C))}{h.M_2}.
\]
This proves the necessary direction of the theorem.

Conversely, suppose $m$ is a positive integer and let $D$ be defined by a nonzero section of $\I_{C/X}(m)$. 
If $D$ is not supported on any components of $N_2$, then we have an exact sequence of sheaves on $D$
\[
0 \to \O_D \to \pi_*\O_{\tilde{D}} \to \K \to 0,
\]
where $K$ is supported on the zero dimension scheme $D\cap N_2$.
By the choice of $m$, we see that $g(D) = g(\tilde{D})$ by the same computation as above. 
Therefore $\chi(\O_D) = \chi(\O_{\tilde{D}})$ and $\chi(\K) = 0$. 
Since $\K$ is supported on a dimension zero subscheme it follows that $\K = 0$.
We conclude that $\tilde{D}\to D$ is an isomorphism.
\end{proof}

Note that we do not write $\tilde{C}$ since $C$ could be supported on a component of $N_2$ although it is almost Cartier, in which case the strict transform $\tilde{C}$ is undefined.


\begin{corollary}
For an almost Cartier divisor $D$ on $X$, there is at most one integer $m$ where $[D+\O_X(m)]$ contains a preserved curve. 
\end{corollary}
\begin{proof}
Suppose $[D_1] = [D+\O_X(n_1)]$ and $[D_2] = [D+\O_X(n_2)]$ are two classes that contain preserved curves $D_1$ and $D_2$. Then for $l\gg 0$, the class $[\O_X(l)-D_1]$ contains an almost Cartier curve $C$. Since $C$ is linked to both $D_1$ and $D_2$, we must have $n_1 = n_2$ by Theorem \ref{unique} (1).
\end{proof}

\begin{corollary}
Let $X\subset\P^3$ be an integral hypersurface with ordinary surface singularities.
If the singular locus of $X$ is irreducible, then any two smooth curves in the same biliaison class on $X$ are linearly equivalent.
\end{corollary}
\begin{proof}
This is true for two smooth curves not supported on the singular locus $N_2$ by the previous corollary.
If two smooth curves are supported on $N_2$, since $N_2$ is irreducible, the two curves must be the same.
\end{proof}

The situation is very different on a smooth projective surface $S$ with an ample divisor $h$. 
If $C$ is any divisor on $S$, then for any $m\gg 0$, the linear system $|C+mh|$ is basepoint-free and contains a smooth curve by Bertini's theorem.

\subsection{Homological invariants}\

In this subsection, we show that certain homological invariants of preserved curves and curves linked to preserved curves on the singular surface $X$ can be computed on the normalization $S$.

\begin{lemma}\label{reduced}
Let $C$ be a curve with an effective line bundle $\F$.
If $C$ is reduced, then $h^0(\F^{-1}) = 0$.
\end{lemma}
\begin{proof}
Let $p:\tilde{C} \to C$ be the normalization, then $\tilde{C}$ is a disjoint union of nonsingular curves. 
Since $p^* \F$ is effective on each component of $\tilde{C}$, it follows that $h^0(\tilde{C}, (p^* \F)^{-1}) = 0$. There is an injection of $\O_C \hookrightarrow p_* \O_{\tilde{C}}$ and therefore 
\[
H^0(C,\F^{-1}) \hookrightarrow H^0(C, (p_* \O_{\tilde{C}}) \otimes \F^{-1}) \cong H^0(\tilde{C}, (p^* \F)^{-1}) = 0. \qedhere
\]
\end{proof}

The assumption that $C$ is of pure dimension 1 and reduced cannot be dropped, as the following two counter-examples demonstrate. (1) A line with an embedded point has sections in infinitely many negative degrees. (2) Let $E$ be the exceptional curve of the blowup of $\P^2$ at origin, then any curve $D$ in $|3E|$ is non-reduced. If $H$ is the very ample line bundle of conics through the origin, then $h^0(\O_D(-H))\ne 0$ by a simple computation.

\begin{corollary}
Let $S$ be a smooth surface with a very ample line bundle $\L$ whose class in $A^1(S)$ is $h$. 
If $h^0(\L) = 4$ or $h^0(\L)\ge 6$, then $h^1(\L^n) = h^2(\L^n) = 0$ for $n> h.h-4$.
\end{corollary}
\begin{proof}
If $h^0(\L) = 4$, then $S$ can be embedded as a hypersurface in $\P^3$ and $h^1(\L^n) = 0$ for all $n$.
If $h^0(\L) \ge 6$, then a general choice of four sections gives map $f:S\to \P^3$ that satisfies the \Cref{Formulas}.
Let $X$ be the image hypersurface and consider the short exact sequence
\[
0 \to \O_X \to f_* \O_S \to \omega_{N_2}(4-h.h) \to 0.
\]
The long exact sequence of cohomologies yields
\[
0 \to H^1((f_*\O_S)\otimes \O_X(n)) \to H^1(\omega_{N_2}(4-h.h+n)).
\]
Since $f_*(\O_S\otimes f^*\O_X(n)) = (f_*\O_S)\otimes \O_X(n)$ and $f$ is affine, it follows that the left term is just $H^1(\L^n)$.
Since $N_2$ is Cohen-Macaulay, the right term is dual to $H^0(\O_{N_2}(-n+h.h-4))$, which vanishes if $N_2$ is reduced and $n>h.h-4$ by the previous lemma.
Since $X$ is a hypersurface of degree $h.h$, we have $h^2(\O_X(n)) = h^0(\O_X(-n+h.h-4))$, which vanishes if $n>h.h-4$.
It follows from the long exact sequence of cohomologies that $h^2(\L^n) = $ for $n>h.h-4$ since $\omega_{N_2}$ has one-dimesional support.
\end{proof}

\begin{proposition}\label{specialty}
If $D$ is a preserved curve on $X$ and $h^1(\O_S(nh))=0$ for some $n$, then
$h^1(\O_D(n))=h^2(\O_S(nh-\tilde{D}))-h^2(\O_S(nh))$.
\end{proposition}
\begin{proof}
We have an exact sequence
\[
H^1(\O_S(nh)) \to H^1(\O_{\tilde{D}}(nh)) \to H^2(\I_{\tilde{D}/S}(nh)) \to H^2(\O_S(nh)) \to 0.
\]
Note that $H^1(\O_{\tilde{D}}(nh)) = H^1(\O_D(n))$ since $D$ is preserved. 
\end{proof}

\begin{proposition}\label{Hilbert}
If $C$ is a curve linked to a preserved curve $D$ on $X$ by $\O_X(m)$ and $h^1(\O_S((h.h+m-n-4)h) = 0$, then
\[
h^0(\I_{C/\P^3}(n)) = h^0(\I_T(n))+h^0(\O_S(nh-\pi^*C-M_2))-h^0(\O_S((n-m)h-M_2)).
\]
Here $T$ is the $(h.h,m)$-complete intersection in $\P^3$ linking $C$ and $D$, where
\[
h^0(\I_T(n)) =\binom{n-m+3}{3}+\binom{n-h.h+3}{3}-\binom{n-m-h.h+3}{3}.
\]
\end{proposition}
\begin{proof}
By liaison theory \cite[Prop 1.2]{MDP}, we have
\[
h^0(\I_{C/\P^3}(n)) = h^0(\I_T(n))+h^1(\O_D(h^2+m-n-4)).
\]
From the Koszul complex
\[
0 \to \O_{\P^3}(-m-h^2) \to \O_{\P^3}(-m)\oplus \O_{\P^3}(-h^2) \to \I_T \to 0
\]
it follows that 
\[
h^0(\I_T(n)) = \binom{n-m+3}{3}+\binom{n-h^2+3}{3}-\binom{n-m-h^2+3}{3}.
\]
By Proposition \ref{specialty}, if $h^1(\O_S((h^2+m-n-4)h)) = 0$, then 
\[
h^1(\O_D(h^2+m-n-4)) = h^2(\O_S((h^2+m-n-4)h-\tilde{D}) - h^2(\O_S((h^2+m-n-4)h)
\]
\[
= h^0(\O_S(\tilde{D}-(h^2+m-n-4)h-c_1)) - h^0(\O_S(-(h^2+m-n-4)h-c_1))
\]
\[
= h^0(\O_S(nh-\pi^*C-M_2))-h^0(\O_S((n-m)h-M_2)). \qedhere
\]
\end{proof}

\begin{proposition}\label{Rao}
Let $C$ be a preserved curve on $X$ that is linked to a preserved curve $D$ by $\O_X(m)$.
If $h^1(\O_S(lh)) = 0$ for all $l$, then 
\begin{align*}
h^1(\I_{C/\P^3}(n)) = & h^0(\O_S(nh))-h^0(\O_S(nh-\tilde{C}))+h^1(\O_S(nh-\tilde{C}))-h^0(\O_T(n))\\
& +h^0(\O_S(nh-\tilde{C}-M_2))-h^0(\O_S((n-m)h-M_2))
\end{align*}
Here $T$ is a $(m,h.h)$-complete intersection curve in $\P^3$ as before.
\end{proposition}
\begin{proof}
Since $h^1(\O_S(lh)) = 0$ for all $l$ by assumption, the formulas for $h^0(\I_{C/\P^3}(n))$ and $h^2(\I_{C/\P^3}(n)) = h^1(\O_C(n))$ are given by the previous two propositions.
Thus we have
\begin{align*}
h^1(\I_{C/\P^3}(n)) & = h^0(\O_C(n))-h^0(\O_{\P^3}(n))+h^0(\I_{C/\P^3}(n))\\
& = h^0(\O_C(n))-h^0(\O_{\P^3}(n))+h^0(\I_T(n))+h^1(\O_D(h.h+m-n-4))\\
& =h^0(\O_C(n))-h^0(\O_T(n))+h^1(\O_D(h.h+m-n-4)).
\end{align*}
Since $h^1(\O_S(l)) = 0$ for all $l$, we have 
\[
h^0(\O_C(n))= h^0(\O_C(nh)) = h^0(\O_S(nh))-h^0(\I_{\tilde{C}/S}(nh))+h^1(\I_{\tilde{C}/S}(nh)).
\]
We also rewrite $h^1(\O_D(h.h+m-n-4))$ as in the proof of the previous proposition.
\end{proof}


The following useful proposition by Gruson-Peskine allows us to compute the dimension of the sections of the normal bundle of a smooth curve on the singular surface $X$.
By deformation theory, this is equal to the dimension of the tangent space of the Hilbert scheme at the closed point corresponding to the curve.

\begin{proposition}[Gruson-Peskine]\label{Normal}
Let $C$ be a smooth connected curve on $X$ not supported on $N_2$, whose proper transform $\tilde{C}$ avoids $R_1$. Then there is an exact sequence of bundles
\[
0 \to \N_{\tilde{C}/S} \to \N_{C/\P^3} \to \O_S(4h-c_1)\otimes \O_{\tilde{C}} \to 0.
\]
\end{proposition}
\begin{proof}
Technically we should pullback all sheaves to $\tilde{C}$ or pushforward all sheaves to $C$, but we omit this from the notations since $\pi:\tilde{C}\to C$ is an isomorphism.
The map of sheaves
\[
\T_S\otimes \O_{\tilde{C}} \to \T_X\otimes \O_C
\]
is injective and locally split since the map $S\to X$ is an immersion away from $R_1$. Thus the dual morphism
\[
\Omega_X \otimes \O_C \to \Omega_S\otimes \O_{\tilde{C}}
\]
is surjective. 
On the other hand, we have a surjection of sheaves 
\[
\Omega_{\P^3} \otimes \O_C \to \Omega_X\otimes \O_C
\]
and thus the composition 
\[
\Omega_{\P^3} \otimes \O_C \to \Omega_S\otimes \O_{\tilde{C}}
\]
is surjective. 
Applying the snake lemma to the diagram of exact sequences
\[
\xymatrix{
0 \ar[r] & \I_{C/\P^3}\otimes \O_C \ar[r] \ar[d]& \Omega_{\P^3}\otimes \O_C  \ar[r] \ar[d] & \Omega_C \ar[r] \ar@{=}[d] & 0\\
0 \ar[r] & \I_{\tilde{C}/S}\otimes \O_{\tilde{C}}\ar[r] & \Omega_S\otimes \O_{\tilde{C}} \ar[r] & \Omega_{\tilde{C}} \ar[r] & 0,
}
\]
we conclude that $\I_{C/\P^3}\otimes \O_C \to \I_{\tilde{C}/S}\otimes \O_{\tilde{C}}$ is a surjection of bundles. 
Dualizing, we obtain a locally split exact sequence of bundles
\[
0\to \N_{\tilde{C}/S} \to \N_{C/\P^3} \to \N_1 \to 0.
\]
Taking top wedge power, we see that 
\[
\wedge^2 \N_{C/\P^3} = \N_{C/S}\otimes \N_1.
\]
Taking third wedge power of the conormal sequence
\[
0 \to \N_{C/\P^3}^* \to \Omega_{\P^3}\otimes \O_C \to \Omega_C \to 0,
\]
it follows that 
\[
\O_C(-4) = \omega_{\P^3}\otimes \O_C = \wedge^3 (\Omega_{\P^3}\otimes \O_C) =  \wedge^2 \N_{C/\P^3}^* \otimes \omega_C. 
\]
Combining, we arrive at
\[
\N_1 = (\wedge^2 \N_C) \otimes \N_{C/S}^* = \I_{C/S}\otimes \omega_C(4) = \I_{C/S} \otimes \O_S(C)\otimes \omega_S \otimes \O_C(4) = \O_S(4h-c_1)\otimes \O_C. \qedhere
\]
\end{proof}

\section{Projections of curves on scrolls}\label{Section2}
In this section, we use results of \Cref{Section1} to study rational normal scrolls $S(a,b)\subset \P^{a+b+1}$ and their general linear projections $f:S(a,b)\to \P^3$.
We refer the readers to \cite{DE} and \cite[\S V.2]{AG} for basic facts of rational normal scrolls and ruled surfaces. 

\bigskip

Let $a\le b$ be two positive integers and let $\E$ be the rank two bundle $\O_{\P^1}(a)\oplus\O_{\P^1}(b)$ on $\P^1$. 
Then $p:\Proj(\op{Sym} \E)\to \P^1$ is a ruled surface over $\P^1$. 
The surface $S:=\Proj(\op{Sym} \E)$ has a tautological bundle $\O_S(1)$ which is very ample and embeds $S$ into $\P^{a+b+1}$ with image $S(a,b)$. 
The surface $S$ is isomorphic to the Hirzebruch surface $H_e$ where $e := b-a$.
Let $\eta$ denote the unique $(-e)$-curve on $S$ and let $f$ denote a fiber of $p:S\to \P^1$, then the Chow ring of $S$ is given by
\[
A(S) = A(\P^1)[\eta]/(\eta^2-c_1(\E)\eta+c_2(\E)) = \mathbb{Z}[f,\eta]/(\eta^2+(b-a)f\eta,f^2).
\]
In the following, we express the divisors in the coordinates given by the basis $\{\eta,f\}$ whenever convenient and use the shorthand $\O_S(c,d)$ for $\O_S(c\eta+df)$.
In particular, the class of $O_S(1)$ is $h = \eta+bf$.
From the exact sequences
\[
0 \to \T_{S/\P^1} \to \T_S \to \T_{\P^1} \to 0
\]
\[
0 \to \O_S \to (p^*\E^*)(1) \to \T_{S/\P^1} \to 0,
\]
we conclude that
\begin{align*}
c(\T_S)&  = p^* c(\T_{\P^1})c((p^*\E^*)(1)) \\
& = (1+2f)(1+(b-a)f+2\eta+(b-a)f\eta+\eta^2)\\
& = (1+2f)(1+(b-a)f+2\eta)\\
& = 1 + (b-a+2)f+2\eta + 4\eta f.
\end{align*}
Therefore $c_1 = 2\eta+(b-a+2)f$ and $c_2 = 4\eta f$.

\begin{proposition}\label{Count}
Suppose $(a,b)\ne (1,1)$.
Let $f:S\to \P^3$ be a general linear projection of $S(a,b)\subset \P^{a+b+1}$ with image $X$.
Then  
\begin{align*}
\deg N_2 &= \frac{1}{2}(b+a-2)(b+a-1),\\
g(N_2) &= \frac{1}{6}(b+a-3)(b+a-4)(2b+2a-1),\\
\deg N_3 &= \frac{1}{3}(b+a-2)(b+a-3)(b+a-4),\\
\deg R_1 &= 2b+2a-4.
\end{align*}
The curve $N_2$ is integral of degree $\deg N_2$ and genus $g(N_2)$, and the number of singular points of $N_2$ is equal to $\deg N_3$.
The curve $M_2$ is reduced and connected and maps 2-1 to $N_2$, ramified over $2e+4n$ pinch points of $X$.
\end{proposition}
\begin{proof}
This follows from a straightforward substitution of the above calculations into the formulas of \Cref{Formulas} .
\end{proof}

\begin{proposition}\label{ACM}
Let $S = S(a,b)\subset \P^{a+b+1}$, then $H^1(\I_S(l)) = H^2(\I_S(l)) = H^1(\O_S(l)) = 0$ for all $l\in\mathbb{Z}$.
\end{proposition}
\begin{proof}
The ideal $I$ of $S(a,b)$ in $\P^{a+b+1}$ is defined by the maximal minors of the matrix
\[
\begin{pmatrix}
x_0 & \dots & x_{a-1} & y_0 & \dots & y_{b-1}\\
x_1 & \dots & x_a & y_1 & \dots & y_b
\end{pmatrix}
\]
in the ring $R:=\C[x_0,\dots, x_a,y_1,\dots, y_b]$.
Since the matrix is $1$-generic, the $R$-module $R/I$ admits a minimal free $R$-resolution given by the Eagon-Northcott complex which is of length $a+b-1$.
In particular, it follows from local duality that 
\[
H^i_m(R/I) = \Ext^{a+b+2-i}_R(R/I, R(-a-b-2))^\vee = 0,\quad i = 1,2.
\]
Here $H^i_m(-)$ denotes the $i$-th local cohomology supported on the irrelevant ideal $m$ of $R$.
The local-to-sheaf exact sequence gives $H^1_m(R/I) \cong \bigoplus_{l\in\mathbb{Z}} H^1(\I_{S}(n))$ and $H^2_m(R/I) \cong \bigoplus_{l\in\mathbb{Z}}H^1(\O_S(l))$.
We conclude that $H^1(\I_S(l)) = H^1(\O_S(l)) = 0$ for all integers $l$.
It also follows that $H^2(\I_S(l)) = 0$ for all $l$ by the short exact sequence 
\[
0 \to \I_S \to \O_{\P^{a+b+1}} \to \O_S\to 0.\qedhere
\]
\end{proof}

\begin{lemma}\label{Natural}
A nontrivial divisor $D = c\eta+df$ on $S(a,b)$ is effective iff $c\ge 0, d\ge 0$.
An effective divisor $D$ has natural cohomology, i.e. $h^1(\O_S(D)) = h^2(\O_S(D)) = 0$, if and only if $d\ge c(b-a)-1$.
\end{lemma}
\begin{proof}
Suppose $D$ is effective and $c<0$, then $p(D)$ is a finite set of points on $\P^1$ since $D.f<0$. It follows that $D$ is concentrated on finitely many fibers, and is thus linearly equivalent to $df$ for some $d>0$.
This is a contradiction to $c<0$. 
It follows that effective divisors must have $c\ge 0$. 

Suppose $D = c\eta+df$ is divisor where $c\ge 0$.
Since $D.f\ge 0$, Grauert's theorem implies that $R^ip_*(\L(D)) = 0$ for all $i>0$. 
Theorefore the Lerray spectral sequence degenerates and $H^i(\L(D)) \cong H^i(\pi_* \L(D))$. 
By projection formula, we have
\begin{align*}
h^i(\pi_* \L(D)) & = h^i(\pi_* (\O_S(c)\otimes \pi^*\O_{\P^1}(d))) \\
&= h^i((\Sym^c \E )\otimes \O_{\P^1}(d)) \\
&= \sum_{i = 0}^c h^i(\O_{\P^1}(-i(b-a)+d)).
\end{align*}
It follows that $D$ is effective iff $d\ge 0$.
It also follows that an effective divisor $D$ has natural cohomology iff $d-c(b-a)>-2$.
\end{proof}


%


\begin{theorem}\label{Normal}
Let $C$ be a smooth connected curve on $X$ not supported on $N_2$ and whose proper transform $\tilde{C}$ avoids $R_1$. Suppose $\tilde{C}$ is of class $c\eta+df$, then $h^0(\tilde{N}_{C/\P^3})$ and $h^1(\tilde{N}_{C/\P^3})$ can be computed explicitly in terms of $a,b,c,d$. In particular, if $c \le 3$ or $d<4b$, then $h^1(\N_{C/\P^3}) = 0$ and $C$ corresponds to a smooth point in the Hilbert scheme of curves in $\P^3$.
If $c\ge 4$ and $d\ge c(b-a)-1+4a$, then 
\begin{align*}
h^0(\N_{C/\P^3}) &=  \frac{1}{2}( a c^{2}- b c^{2}+ a c- b c)+c d+6 a+6 b+c+d-3\\
& = \dim |C|+6a+6b-3-4\deg C.\\
h^1(\N_{C/\P^3}) &= \frac{1}{2}(ac^2  - bc^2  - 7ac - bc) + cd + 6a + 6b + c - 3d - 3\\
& = \dim |C|+6a+6b-3.
\end{align*}
\end{theorem}
\begin{proof}
Since $\tilde{C}$ is a smooth curve, we have $d\ge c(b-a)$ by \cite[Cor V.2.18]{AG}.
\Cref{Natural} implies that $\O_S(\tilde{C})$ has natural cohomology.
The short exact sequence $0 \to \O_S \to \O_S(\tilde{C}) \to \O_{\tilde{C}}\otimes\O_S(\tilde{C})\to 0$ implies that $H^1(\N_{\tilde{C}/S}) = H^1(\O_{\tilde{C}}\otimes\O_S(\tilde{C}))=0$ since $H^1(\O_S) =H^2(\O_S) = 0$.
The short exact sequence from \Cref{Normal} $0 \to \N_{\tilde{C}/S} \to \N_{C/\P^3} \to \O_S(4h-c_1)\otimes \O_{\tilde{C}}\to 0$
thus implies that $H^1(\N_{C/\P^3}) = H^1(\O_S(4h-c_1)\otimes \O_{\tilde{C}})$ since the sheaves are supported on a curve.
Finally, consider the short exact sequence
\[
0 \to \I_{\tilde{C}/S}(4h-c_1) \to \O_S(4h-c_1) \to \O_S(4h-c_1)\otimes \O_{\tilde{C}} \to 0.
\]
The divisor $4h-c_1$ has coordinates $(2,3b+a-2)$ in the basis $\{\eta,f\}$, which is effective with natural cohomology by \Cref{Natural}.
We conclude that 
\[
H^1(\N_{C/\P^3}) \cong H^1(\O_S(4h-c_1)\otimes \O_{\tilde{C}} ) \cong  H^2(\I_{\tilde{C}/S}(4h-c_1)) \cong H^0(\O_S(C-4h))^\vee.
\]
If either $c<4$ or $d<4b$, then $h^0(\O_S(C-4h)) = 0$.
If $c\ge 4$, then 
\[
h^0(\O_S(C-4h)) = h^0(\O_S(c-4,d-4b))= \sum_{i = 0}^{c-4} h^0(\O_{\P^1}(-i(b-a)+d-4b)).
\]
Since $\chi(\N_{C/\P^3}) = 4\deg C$, it follows that 
\[
h^0(\N_{C/\P^3}) = 4(d+ac)+\sum_{i = 0}^{c-4} h^0(\O_{\P^1}(-i(b-a)+d-4b)).
\]
If $d\ge c(b-a)-1+4a$, then $\O_S(c-4,d-4b)$ is effective with natural cohomology by \Cref{Natural}.
We apply Riemann-Roch formula on $S$ to obtain
\[
h^1(\N_{C/\P^3}) = \chi(\O_S(C-4h))= \frac{1}{2}(ac^2  - bc^2  - 7ac - bc) + cd + 6a + 6b + c - 3d - 3.\]
Recall that the intersection formula states that
\[
(-C).4h = \chi(\O_S)-\chi(\O_S(C))-\chi(-4h)+\chi(\O_S(C-4h)).
\]
It follows that $h^1(\N_{C/\P^3}) = \dim |C|+6a+6b-3- 4\deg C$ and $h^0(\N_{C/\P^3}) = \dim |C| + 6a+6b-3$.
\end{proof}

The dimension of the family of surfaces $X$ obtained from linear projections of $S(a,b)\subset \P^{a+b+1}$ into $\P^3$ is given by
\[
\dim \op{\mathbb{G}r}(a+b-3,a+b+1) + \dim \op{PGL}(3,\mathbb{C})- \dim \op{Aut}(S(a,b)).
\] 
Recall that there is an exact sequence of groups for the Hirzebruch surface $H_e$
\[
0 \to H^0(\O_{\P^1}(e))\rtimes \C^* \to \op{Aut}(H_e) \to \op{PGL}(2,\C)\to 0.
\]
In particular, we deduce that $\dim \op{Aut}(S(a,b)) = b-a+5$.
It follows that there is a $(5a+3b+2)$-dimension family of integral surfaces of degree $a+b$ in $\P^3$ arising as general linear projections of the scroll $S(a,b)\subset\P^{a+b+1}$.
If we vary the curve $\tilde{C}$ in the linear system as well, we end up with a family of curves in $\P^3$ of dimension $\dim |C|+5a+3b+2$.
Suppose linear system $\tilde{C} = c\eta+df$ satisfies $c\ge 4$ and $d\ge c(b-a)-1+4a$, then the difference between $h^0(\N_{C/\P^3})$ and the dimension of the family of the curves is $a+3b-5$, which does not depend on the class of $\tilde{C}$.
There are two possibilities in this situation.
Either the family of curves are not dense in the component of the Hilbert scheme they belong to, or there is a highly nonreduced component of the Hilbert scheme whose general member is given by such a curve.
We are not able to determine which is the case.

\subsection{Maximal rank curves on the ruled cubic surface}\

In this subsection we continue the study of curves on a ruled cubic surface initiated by Hartshorne in \cite{H1}. 

\bigskip

Consider the general projection of $S(1,2)\subset\P^4$ into $\P^3$. 
By Proposition \ref{Count}, the image surface $X$ is a ruled cubic surface with singularity a double line $N_2$ that has $2$ pinch points on it.
The map $\pi:S\to X$ is an isomorphism away from the conic $M_2$ and the line $N_2$, and maps $M_2$ generically 2-1 to $N_2$ branched over the two pinch points of $X$. 

\bigskip

Since all integral ruled cubic surfaces that are not cones all differ by a coordinate change in $\P^3$ \cite{GP}, we can afford to work explicitly.
In doing so, we verify the formulas and see the geometry more clearly.
Let $S$ be the blowup of $\P^2 = \Proj \C[s,t,u]$ at the point $O = V(s,t)$, embedded by the complete linear system $h$ of proper transforms of conics passing through $O$. Then $S$ is parametrized as $[tu:t^2:st:su:s^2]$, and its defining equations in $\P^4 = \Proj \C[x,y,z,w,v]$ are given by the $2\times 2$-minors of the matrix 
\[
M = \begin{bmatrix}x & y & z\\ w& z & v\end{bmatrix}.
\]
Let $\pi:S\to \Proj \C[x,y,w,v]$ be the projection of $S$ away from the point $V(x,y,w,v)$. The image $X$ is parametrized by $[tu:t^2:su:s^2]$ and has the defining equation $x^2v-y^2w$. The double locus $M_2$ on $S$ is defined by $V(x,w) \cap S = [0:t^2:st:0:t^2]$ which maps 2-1 to the line $N_2 = V(x,w) = [0:t^2:0:s^2]$. 
Since $X$ is the image of the ruled surface $S$, $X$ itself is spanned by lines and thus the name ruled cubic surface. 
Let $\sigma:M_2\to N_2$ be the involution defined by 
\[
[0:t^2:st:0:t^2] \mapsto [0:t^2:-st:0:t^2].
\] 
The locus $R_1$ consists of the two ramification points of $M_2\to N_2$, which are $V(x,z,w,v)$ and $V(x,y,z,w)$. 

\bigskip

By \cite[Thm 4.1, 4.5]{HP}, there is an exact sequence of groups 
\[
0 \to \APic(X) \to \Pic(S) \oplus \Cart M_2/\pi^*\Cart N_2 \to \underbrace{\Pic M_2 / \pi^*\Pic N_2}_{ \mathbb{Z}/2\mathbb{Z}} \to 0.
\]
In this way, we describe an almost Cartier divisor on $X$ by a triple $(c,d, \alpha)$, where $(c,d)$ is the class $c\eta+df$ in $\Pic(S)$ as before and $\alpha$ is a Cartier divisor of $M_2$, subject to the constraint that $d \equiv \deg \alpha \op{mod} 2$. 

\begin{definition}
For any divisor $\alpha\in \Cart M_2$, we set $\overline{\alpha} \in \Cart M_2$ to be the effective divisor of least degree such that $\alpha \equiv \overline{\alpha} \op{mod} \pi^* \Cart N_2$. More explicitly, if $\alpha =\sum_i P_i - \sum_j Q_j$ where $P_i\ne Q_j$, then we remove all the pairs $P+\sigma(P)$ from $\sum_i P_i + \sum_j \sigma(Q_j)$ to obtain $\overline{\alpha}$.
\end{definition}
 
\begin{proposition}[Hartshorne {\cite[Prop 6.5]{H1}}]
A class $(c,d,\alpha) \in \APic(X)$ is effective if and only if one of the following is true:
\begin{enumerate}
    \item $c>0, d>0$, or
    \item $d = \alpha = 0, c>0$, or
    \item $c=0$, $d > 0$ and $\deg \overline{\alpha} \le d$. 
\end{enumerate}
\end{proposition}

Note that the basis $\{f+\eta,-\eta\}$ instead of $\{\eta,f\}$ was used for $\Pic(S)$ in \cite{H1}.

\bigskip

With the effectiveness criterion in hand, it is relatively easy to classify classes that contain a preserved curve on $X$. 

\begin{proposition}
A class $(c,d,\alpha)$ contains a preserved curve if and only if one of the following is true:
\begin{enumerate}
    \item $c >0, d>0$ and $d = \deg \overline{\alpha}$, or
    \item $d = \alpha = 0, c>0$, or
    \item $c = 0, d>0$ and $d = \deg \overline{\alpha}$.
\end{enumerate}
\end{proposition}
\begin{proof}
For (3), take $c$ disjoint points on $M_2$ that contain no pairs of involution points, then the sum of the $c$-fibers through them will be preserved. 
For (2) take a multiple structure on the $(-1)$-curve, which does not intersect $M_2$ and is therefore preserved. 
For (1) take the union of curves in (2) and (3).

Now we argue that (1)-(3) is necessary. 
Suppose $D$ is a preserved curve, then the linear system corresponding to the projection separates points and tangent vectors on $\tilde{D}$. 
Thus $\tilde{D}$ cannot meet pairs of involution points $P_i\ne \sigma(P_i)$.
If $\tilde{D}$ meets a point $Q$ in the branch locus $R_i$, then it must meet it with multiplicity one. 
Otherwise the tangent space of $\tilde{D}$ will contain the tangent line of $\Gamma$ at $Q$, which is the line $\overline{OQ}$ where $O$ is the point of projection.
In this case the linear system $H$ would fail to separate tangent vectors of $\tilde{D}$ at $Q$. 
Taking $\alpha = \tilde{D}\cap M_2$, it is clear that $\deg \overline{\alpha} = \deg \alpha = d$.
\end{proof}

\begin{theorem}
Every smooth connected curve $C$ on $X$ is linked to a preserved curve $D$ except for the line $N_2$ and the $(-1)$-curve $\eta$. 
\end{theorem}
\begin{proof}
Every link of $N_2$ is not almost Cartier, and thus must be supported on $N_2$.
In particular, no such curve can be preserved. 
Let $C$ be a smooth irreducible curve of the class $(c,d,\alpha)$. 
Then either (1) $c>0,d>0, \deg \overline{\alpha} =d$, or (2) $C$ is the exceptional curve where $d = \alpha = 0$ and $c = 1$, or (3) $C$ is a ruling where $c = 0, d>0$ and $\deg \overline{\alpha} = d$. 
The linked divisor $D = dH-C = (d-c,d,\sigma(\alpha))$ contains a preserved curve for case (1) and (3) by the previous proposition.  
\end{proof}

More explicitly, let $C \ne N_2$ be a smooth connected curve in the class $(c,d,\alpha)$. 
If $C$ is not $E$ then $d>0$. 
Its proper transform $\tilde{C}$ meets $M_2$ at $d$ points $\alpha$ on $M_2$ containing no pairs of involution points. 
The preserved curve $D$ can be constructed using the sum of $d$-fibers passing through $\sigma(\alpha)$ and a $(d-c)$-multiple of the exceptional curve $\eta$ which does not meet $M_2$.

\begin{definition}
A closed subscheme $V$ of $\P^n$ is said to have maximal rank if for every $d\ge 0$ the map
\[
H^0(\O_{\P^n}(d)) \to H^0(\O_V(d))
\]
has maximal rank, i.e. either injective or surjective.  
\end{definition}

Note that having maximal rank is the same as having either $H^0(\I_V(d)) = 0$ or $H^1(\I_V(d)) = 0$ for every $d\ge 0$.
Examples of maximal rank varieties include ACM (arithmetic Cohen-Macaulay) curves $C$, which are characterized by the vanishing of $H^1(\I_C(d))$ for all $d$.


\begin{proposition}
Let $C$ be a curve of class $(c,d)$ on $S(1,2)\subset \P^4$, then 
\begin{align*}
& h^1(\I_{C/\P^4}(c-1)) = 0,\\
& h^1(\I_{C/\P^4}(n)) = \sum_{i = 0}^{n-c}h^0(\O_{\P^1}(d+i-2n-2)),\quad  \forall n\ge c, \tag{A}\label{A}\\
& h^1(\I_{C/\P^4}(n)) = \sum_{i =0}^{c-n-2}h^0(\O_{\P^1}(2n+i+1-d)), \quad \forall n\le c-2\tag{B}\label{B}.
\end{align*}
\end{proposition}
\begin{proof}
First we note that there is a short exact sequence
\[
0 \to \I_{S/\P^4} \to \I_{C/\P^4} \to j_*\I_{C/S} \to 0
\]
where $j:S\hookrightarrow \P^4$ is the inclusion.
Since $H^1(\I_{S/\P^4}(n)) = H^2(\I_{S/\P^4}(n)) = 0$ for all $n$ by \Cref{ACM}, we conclude that $H^1(\I_{C/\P^4}(n) \cong H^1(\I_{C/S}(n))$ for all $n$. 
Neither $(c-1)h-C$ nor $C-(c-1)h-c_1$ are effective since their first coordinates are $-1$. 
It follows that 
\begin{align*}
h^1(\I_{C/S}(c-1)) &= h^1(\O_S((c-1)h-C) \\
&= -\chi(\O_S((c-1)h-C)) \\
&= \frac{1}{2} ((c-1)h-C).((c-1)h-C+c_1)+1 = 0.
\end{align*}
If $n\ge c$, then the first coordinate of $nh-C$ is nonnegative and thus $h^1(\I_C(n))$ is given by
\[
h^1(\O_S(n-c,2n-d))= \sum_{i = 0}^{n-c}h^1(\O_{\P^1}(2n-d-i)) = \sum_{i = 0}^{n-c}h^0(\O_{\P^1}(d+i-2n-2))
\]
using the same computations in \Cref{Natural}.
If $n\le c-2$, then the first coordinate of $C-nh-c_1$ is nonnegative and thus
\begin{align*}
 h^1(\I_{C/S}(n)) & = h^1(\O_S(nh-C))\\
 & = h^1(\O_S(C-nh-c_1))\\
 & = h^1(\O_S(c-n-2,d-2n-3))\\
 & =  \sum_{i = 0}^{c-n-2} h^1(\O_{\P^1}(d-2n-3-i)) \\
 & = \sum_{i =0}^{c-n-2}h^0(\O_{\P^1}(2n+i+1-d)).\qedhere
\end{align*}
\end{proof}

\begin{corollary}
Let $C$ be a curve of class $(c,d)$ on $S(1,2)$. 
\begin{enumerate}
\item If $d\ge 2c+2$, then $h^1(\I_{C/\P^4}(n))$ is nonzero and strictly decreasing on the interval $[c,d-c-2]$ with value (\ref{A}) and vanishes elsewhere.
\item If $b<2c+2$, then $h^1(\I_{C/\P^4}(n))$ is nonzero and strictly increasing on the interval $[d-c+1,c-2]$ with value (\ref{B}) and vanishes elsewhere.
\end{enumerate}
In particular, $C$ is ACM iff $2c-2\le d\le 2c+1$.
\end{corollary}


\begin{proposition}
Apart from $N_2$, which is obviously of maximal rank, a smooth curve $C$ of the class $(c,d,\alpha)$ on $X$ has maximal rank iff $(c,d)$ is one of the following:
\[
\{(1,0),(1,1),(1,2),(2,2),(2,3),(2,4),(3,3),(3,4)\}.
\]
Among them only $(1,0)$, $(1,1)$ and $(1,2)$ are ACM. 
\end{proposition}
\begin{proof}
Let $C$ be a smooth curve on $X$ that is not $N_2$.
Suppose $c>3$.
Since $\tilde{C}$ is smooth, it follows that $d\ge c$ by \cite[Cor V.2.18]{AG}.
The expressions of \Cref{Rao} simplies to
\begin{align*}
h^1(\I_{C/\P^3}(3)) =& h^0(\O_S(3h))-\cancelto{0}{h^0(\O_S(3h-\tilde{C}))}+h^1(\O_S(3h-\tilde{C}))-h^0(\O_T(3))\\
& +\cancelto{0}{h^0(\O_S(3h-\tilde{C}-M_2))}-\cancelto{0}{h^0(\O_S((3-d)h-M_2))}\\
=&22+h^1(\O_S(3-c,6-d))-19 \ge 3.
\end{align*}
Since $C$ lies on a cubic surface, it follows that if $C$ is maximal rank then we must have $c\le 3$.
On the other hand, 
\begin{align*}
h^1(\I_{C/\P^3}(3)) &= h^0(\O_C(3))-h^0(\O_{\P^3}(3))+h^0(\I_{C/\P^3}(3))\\
& \ge \chi(\O_C(3))-20+h^0(\I_{C/\P^3}(3)).
\end{align*}
It remains to estimate $H^0(\O_C(3))$.
By Riemann-Roch, we have
\[
h^0(\O_C(3)) \ge \chi(\O_C(3)) = 3 \tilde{C}.h+1-g(C) = \frac{1}{2} c^{2}+\frac{7}{2} c+(4-c)d.
\]
This is an increasing function in $d$, thus if $d\ge 5$ then 
\[
h^0(\O_C(3)) \ge \frac{1}{2} c^{2}-\frac{3}{2} c+20.
\]
Since $c\le 3$, it follows that $h^0(\O_C(3))\ge 20$ unless $c = 1, d = 5$.
But for $(c,d) = (1,5)$, we see that $3h-C$ is effective and thus $C$ lies on another cubic surface, i.e. $h^0(\I_{C/\P^3}(3))\ge 2$.
This shows that if $c\le 3$ and $d\ge 5$ then $h^1(\I_C(3))>0$, therefore a maximal rank curve $C$ must satisfy $c\le 3$ and $d\le 4$.
We compute $h^0(\I_{C/\P^3}(n))$ and $h^1(\I_{C/\P^3}(n))$ by hand using \Cref{Hilbert} and \Cref{Rao} for these finitely many cases and verify the claim.
We omit the computations.
\end{proof}

It follows that for $c>3$ and $2c-2\le d\le 2c+1$, the general projections $C$ of smooth ACM curves in the linear system $|c\eta+df|$ on $S(1,2)$ do not have maximal rank in $\P^3$.
These curves have degree $c+d$ and genus $ -\frac{1}{2}c^2  + cd - \frac{1}{2}c - d + 1$.
Since $d\sim 2c$, we see that $g(C)\sim \frac{1}{6}(\deg C)^2 $ for $c\gg 0$.


\begin{thebibliography}{9}

\bibitem{BE}
Ballico, E.(I-SNS); Ellia, Ph.(F-NICE)
\emph{On the postulation of a general projection of a curve in $\P^3$}.
Ann. Mat. Pura Appl. (4) 142 (1985), 15–48 (1986).

\bibitem{DE}
Eisenbud, David; Harris, Joe.
\emph{On varieties of minimal degree (a centennial account)}. Algebraic geometry, Bowdoin, 1985 (Brunswick, Maine, 1985), 3–13,
Proc. Sympos. Pure Math., 46, Part 1, Amer. Math. Soc., Providence, RI, 1987.


\bibitem{Harris}
P.Griffiths, J.Harris.
\textit{Principles of algebraic geometry}. 
Reprint of the 1978 original. Wiley Classics Library. John Wiley \& Sons, Inc., New York, 1994. xiv+813 pp.

\bibitem{BN}
Griffiths, Phillip; Harris, Joseph
\emph{On the variety of special linear systems on a general algebraic curve}.
Duke Math. J. 47 (1980), no. 1, 233–272.


\bibitem{GP2}
Gruson, Laurent; Peskine, Christian.
\emph{Genre des courbes de l'espace projectif}. (French) Algebraic geometry (Proc. Sympos., Univ. Tromsø, Tromsø, 1977), pp. 31–59,
Lecture Notes in Math., 687, Springer, Berlin, 1978.


\bibitem{GP}
L.Gruson, C.Peskine.
\textit{Genre des courbes de l'espace projectif. II}.
Ann. Sci. \'{E}cole Norm. Sup. (4) 15 (1982), no. 3, 401–418. 



\bibitem{Kleiman}
S.Kleiman.
\textit{The enumerative theory of singularities}.
Translated from the English by V. S. Kulikov. 
Uspekhi Mat. Nauk 35 (1980), no. 6(216), 69–148. 

\bibitem{Kleiman1}
S.Kleiman.
\textit{Multiple-point formulas. I. Iteration}. 
Acta Math. 147 (1981), no. 1-2, 13–49. 

\bibitem{Kleiman2}
S.Kleiman.
\textit{Multiple-point formulas. II. The Hilbert scheme}. 
Enumerative geometry (Sitges, 1987), 101–138, Lecture Notes in Math., 1436, Springer, Berlin, 1990.


\bibitem{KLU2}
S.Kleiman,  J.Lipman, B.Ulrich.
\textit{The source double-point cycle of a finite map of codimension one}. 
Complex projective geometry (Trieste, 1989/Bergen, 1989), 199–212, 
London Math. Soc. Lecture Note Ser., 179, Cambridge Univ. Press, Cambridge, 1992. 

\bibitem{KLU}
S.Kleiman,  J.Lipman, B.Ulrich.
\textit{The multiple-point schemes of a finite curvilinear map of codimension one}.  
Ark. Mat. 34 (1996), no. 2, 285–326. 

\bibitem{KL}
Kleiman, Steven L.; Laksov, Dan
\emph{On the existence of special divisors}.
Amer. J. Math. 94 (1972), 431–436.

\bibitem{Kleppe}
Kleppe, Jan O.(N-OCE)
\emph{Nonreduced components of the Hilbert scheme of smooth space curves}. Space curves (Rocca di Papa, 1985), 181–207,
Lecture Notes in Math., 1266, Springer, Berlin, 1987.

\bibitem{Gieseker}
Gieseker, D.
\emph{Stable curves and special divisors: Petri's conjecture}.
Invent. Math. 66 (1982), no. 2, 251–275.

\bibitem{AG}
R.Hartshorne.
\textit{Algebraic geometry}. 
Graduate Texts in Mathematics, No. 52. Springer-Verlag, New York-Heidelberg, 1977. 

\bibitem{H1}
R.Hartshorne.
\textit{Generalized Divisors on Gorenstein Schemes}.
K-theory 8: 287-339, 1994.

\bibitem{H2}
Hartshorne, Robin.
\emph{Stable vector bundles of rank 2 on $\P^3$}.
Math. Ann. 238 (1978), no. 3, 229–280. 

\bibitem{Nonsmooth}
Hartshorne, Robin(J-KYOT-R)
\emph{Une courbe irr\'{e}ductible non lissifiable dans $\P^3$}. (French. English summary) [An irreducible nonsmoothable curve in P3]
C. R. Acad. Sci. Paris Sér. I Math. 299 (1984), no. 5, 133–136.

\bibitem{HP}
R.Hartshorne, C.Polini.
\textit{Divisor Class Group of Singular Surfaces}.
Trans. Amer. Math. Soc. 367 (2015), no. 9, 6357–6385. 

\bibitem{Larson}
E.Larson.
\text{The Maximal Rank Conjecture}.
https://arxiv.org/abs/1711.04906

\bibitem{MDP}
M.Martin-Deschamps, D.Perrin.
\textit{Sur la classification des courbes gauches}.
Astérisque No. 184-185 (1990), 208 pp. 

\bibitem{Tour}
E.Mezzetti, D.Portelli.
\textit{A tour through some classifcal theorems on algebraic surfaces}.
An. Stiint. Univ. Ovidius Constanta Ser. Mat. 5 (1997), no. 2, 51–78. 

\bibitem{Mumford}
Mumford, David
\emph{Further pathologies in algebraic geometry}.
Amer. J. Math. 84 (1962), 642–648

\bibitem{PS}
Peskine, C.; Szpiro, L.
\emph{Liaison des var\'{e}t\'{e}s alg\'{e}briques. I.} (French)
Invent. Math. 26 (1974), 271–302.

\bibitem{JR}
J.Roberts.
\textit{Hypersurfaces with Nonsingular Normalization and Their Double Loci}.
J. Algebra 53 (1978), no. 1, 253–267. 

\bibitem{JR1}
J.Roberts.
\textit{Generic projections of algebraic varieties}. 
Amer. J. Math. 93 1971 191–214. 
\end{thebibliography}
\end{document}